\documentclass[a4paper,twoside,12pt]{article} %

\usepackage[margin=3cm,a4paper]{geometry} 
\usepackage{amsmath,amssymb,mathtools,lineno,graphicx,textcomp,booktabs,url,setspace,xcolor,soul,eurosym}
\usepackage{pgf,tikz,color}
\usetikzlibrary{arrows}

\usepackage{amsmath, amssymb, amsthm, bm}
\usepackage{color}
\usepackage{url, hyperref}
\usepackage{enumerate,verbatim}
\usepackage{graphicx}
\usepackage{appendix}
\usepackage{cite}

\usepackage{algorithm, algorithmic}

\newtheorem{thm}{Theorem}
\newtheorem*{mtheorem*}{Main Result}

\newtheorem{lem}[thm]{Lemma}
\newtheorem{cor}[thm]{Corollary}
\numberwithin{thm}{section}

\theoremstyle{definition}

\renewcommand{\qedsymbol}{$\blacksquare$}

\theoremstyle{remark}
\newtheorem{rem}[thm]{Remark}

\theoremstyle{definition}

\usepackage{thmtools}
 
\declaretheoremstyle[%
  spaceabove=-6pt,%
  spacebelow=6pt,%
  headfont=\normalfont\itshape,%
  postheadspace=1em,%
  qed=\qedsymbol%
]{mystyle}

\newcommand{\m}[1]{\mathcal{#1}}

\mathchardef\mh="2D

\title{Minimal multiple blocking sets}
\author{Anurag Bishnoi, Sam Mattheus and Jeroen Schillewaert}

\begin{document}
\maketitle

\begin{abstract}
We prove that a minimal $t$-fold blocking set in a finite projective plane of order $n$ has cardinality at most \[\frac{1}{2} n\sqrt{4tn - (3t + 1)(t - 1)} + \frac{1}{2} (t - 1)n + t.\]
This is the first general upper bound on the size of minimal $t$-fold blocking sets in finite projective planes and it generalizes the classical result of Bruen and Thas on minimal blocking sets. 
From the proof it directly follows that if equality occurs in this bound then every line intersects the blocking set $S$ in either $t$ points or $\frac{1}{2}(\sqrt{4tn  - (3t + 1)(t - 1)}  + t - 1) + 1$ points.
We use this to show that for $n$ a prime power, equality can occur in our bound in exactly one of the following three cases: (a) $t = 1$, $n$ is a square and $S$ is a unital; (b) $t = n - \sqrt{n}$, $n$ is a square and $S$ is the complement of a Baer subplane; (c) $t = n$ and $S$ is equal to the set of all points except one. 
For a square prime power $q$ and $t \leq \sqrt{q} + 1$, we give a construction of a minimal $t$-fold blocking set $S$ in $\mathrm{PG}(2,q)$ with $|S| = q\sqrt{q} + 1 + (t - 1)(q - \sqrt{q} + 1)$. 
Furthermore, we obtain an upper bound on the size of minimal blocking sets in symmetric $2$-designs and use it to give new proofs of other known results regarding tangency sets in higher dimensional finite projective spaces. 
We also discuss further generalizations of our bound. 
In our proofs we use an incidence bound on combinatorial designs which follows from applying the expander mixing lemma to the incidence graph of these designs. 
\end{abstract}

\section{Introduction}
\textit{Blocking sets} are sets of points in a finite projective or affine plane  that intersect every line in at least one point.
These objects have been extensively studied over the years. 
While typically the focus is on the minimum size of a blocking set, upper bounds on \textit{minimal} blocking sets have also been studied; a blocking set $S$ is called minimal if none of its proper subsets forms a blocking set or equivalently if every point of $S$ is contained in a line intersecting $S$ in exactly one point. 

A well known result of Bruen and Thas \cite[Corollary 6]{Bruen-Thas77} states that in a finite projective plane of order $n$ the size of a minimal blocking set is bounded above by $n \sqrt{n} + 1$. This bound is sharp when the plane is Desarguesian and $n$ is a square (so, an even power of a prime), as shown by the points of a non-degenerate Hermitian variety, i.e. a set of points in $\mathrm{PG}(2,q^2)$ which is projectively equivalent to the set of points $(x,y,z) \in \mathrm{PG}(2,q^2)$ satisfying $x^{q+1}+y^{q+1}+z^{q+1}=0$. In fact, if a minimal blocking set has size $n \sqrt{n} + 1$, then every line intersects the blocking set in either $1$ point or $\sqrt{n} + 1$ points, and thus it forms a unital (see \cite{Barwick-Ebert_book} for the definition of unitals and their basic properties). 
For projective planes of non-square order $n$, $n \neq 5$, this bound has been improved to $n \sqrt{n} + 1 - \frac{1}{4}s(1 - s)n$ where $s$ is the fractional part of $\sqrt{n}$ \cite[Theorem 5.1]{Szonyi05}. 
The reader is referred to \cite{Blokhuis96_survey} and \cite{DeBeule-Storme12} for surveys on blocking sets in finite projective spaces, and to \cite{Ball_book} and \cite{Dembowski} for the relevant background in finite geometry.

The  main objective of this paper is to generalize the Bruen-Thas upper bound for blocking sets to the case of \textit{multiple blocking sets}.

A $t$-fold blocking set in a projective plane $\pi$ of order $n$ is a set of points $S$ such that each line of $\pi$ intersects $S$ in at least $t$ points and some line of $\pi$ intersects $S$ in exactly $t$ points. It is called minimal if every point of $S$ is contained in a line intersecting $S$ in exactly $t$ points. 
For $t = 2, 3$, these sets are known as double and triple blocking sets, respectively. 
Multiple blocking sets were introduced by Bruen in \cite{Bruen83} and general lower bounds on these sets were provided by Ball \cite{Ball96}.
As far as we know, there are no non-trivial general upper bounds on \textit{minimal} multiple blocking sets. The following is our main result. 

\begin{thm}
\label{thm:main}
A minimal $t$-fold blocking set $S$ in a finite projective plane $\pi$ of order $n$ has size at most \[\frac{1}{2} n\sqrt{4tn - (3t + 1)(t - 1)} + \frac{1}{2} (t - 1)n + t.\]
If the size of $S$ is equal to this upper bound, then every line of $\pi$ intersects $S$ in exactly $t$ or $\frac{1}{2}(\sqrt{4tn  - (3t + 1)(t - 1)}  + t - 1) + 1$ points. Moreover, if $n$ is a prime power, then equality can only occur in one of the following cases:
\begin{itemize}
\item $t=1$, $n$ a square, and $S$ is a unital in $\pi$. 
\item $t=n-\sqrt{n}$, $n$ a square, and $S$ is the complement of a Baer subplane in $\pi$.
\item $t=n$ for any $n$, and $S$ is the plane $\pi$ with one point removed.
\end{itemize} 
\end{thm}

Note that by taking $t = 1$ in Theorem \ref{thm:main} we recover the result of Bruen and Thas. A line which intersects a set $S$ of points in a plane in $k$ points will be called a \textit{$k$-secant} of $S$, and for $k = 1$ it will be called  a \textit{tangent} of $S$. 


The technique used to prove Theorem \ref{thm:main} involves an incidence bound  on combinatorial designs which is a direct consequence of the bipartite version of the expander mixing lemma. 
We believe that this proof technique is of independent interest and it can potentially be used to prove more results in finite geometry. 
We will discuss this incidence bound  in Section \ref{sec:incidence_bound}. 
The proof of our main result is contained in Sections \ref{sec:proof} and \ref{sec:equality}. 
In Section \ref{sec:alternate}, we give an alternative proof which is in similar spirit to some of the known proofs of the Bruen-Thas upper bound and does not involve spectral techniques. 
In Section \ref{sec:construction} we provide a construction of a minimal $t$-fold blocking set of size $q\sqrt{q} + 1 + (t-1)(q - \sqrt{q} + 1)$ in the Desarguesian projective plane $\mathrm{PG}(2, q)$, for $q$ square and $t \leq \sqrt{q} + 1$, by adding $(t - 1)(q - \sqrt{q} + 1)$ carefully chosen points to a non-degenerate Hermitian variety. The construction of a ``large'' $t$-fold minimal blocking set when $q$ is not a square appears to be hard. 
In Section \ref{sec:higher_dim}, we prove general upper bounds on minimal blocking sets in symmetric designs using the main incidence bound. As a corollary we obtain another result of Bruen and Thas \cite{Bruen-Thas82,Thas74}. 
We also discuss multiple blocking sets with respect to hyperplanes in $\mathrm{PG}(n,q)$. 
In Section \ref{sec:further}, we discuss another generalization of the Bruen-Thas upper bound that can be obtained using our techniques. 
Finally in Section \ref{sec:open} we mention some interesting open problems related to our work. 


\section{The incidence bound}
\label{sec:incidence_bound}
%
%

A \textit{$2$-$(v, k, \lambda)$ design} is a collection $\m B$ of $k$-subsets of a $v$ element set $X$ such that every two distinct elements of $X$ are contained together in exactly $\lambda$ elements of $\m B$. 
The elements of $X$ are called \textit{points} and the elements of $\m B$ are called \textit{blocks}.  
From double counting it follows that the number of blocks through any given point in a $2$-$(v, k, \lambda)$ design is equal to $r \coloneqq \frac{\lambda (v - 1)}{(k - 1)}$, and that the total number of blocks is equal to $b \coloneqq \frac{\lambda v(v - 1)}{k(k-1)}$. 
The parameter $r$ of the design is also known as the replication number. 
A design is called \textit{symmetric} if $b = v$, or equivalently $r = k$. 

The incidence graph of a $2$-$(v, k, \lambda)$ design is the bipartite graph $G = (L, R, E)$ obtained by taking $L$ as the set of points, $R$ as the sets of blocks and taking an edge between a point and a block if the point is contained in the block.
This graph is biregular with $d_L = r$ and $d_R = k$. 
It is well known that the non-zero eigenvalues of this graph are $\sqrt{rk}, \sqrt{r - \lambda}, - \sqrt{r - \lambda}, -\sqrt{rk}$.
By applying the expander mixing lemma \cite[Section 2.4]{Hoory-Linial-Wigderson06} to this  graph we get the following incidence bound, which was first proved by Haemers using interlacing techniques \cite[Section 5]{Haemers95_interlacing} and recently rediscovered by Lund and Saraf \cite[Theorem 1]{Lund-Saraf16}\footnote{Lund and Saraf have given a weaker bound in \cite[Theorem 1]{Lund-Saraf16}, but this stronger version also follows from the expander mixing lemma.}. 

\begin{lem}[Incidence Bound]

\label{lem:incidence_bound}
Let $(X, \m B)$ be a $2 \mh (v, k, \lambda)$ design with total number of blocks $b$ and replication number $r$.
Let $S$ be a subset of $X$ and let $\m L$ be a subset of $\m B$. 
Let $i(S, \m L)$ denote the cardinality of the set $\{(x, L) \in S \times \m B \mid x \in L\}$, i.e.  the number of incidences between the set $S$ and the set $\m L$. 
Then we have 
\[\left\lvert i(S, \m L) - \frac{r|S||\m L|}{b} \right\rvert \leq \sqrt{r - \lambda} \sqrt{|S||\m L|\left(1 - \frac{|S|}{v}\right)\left(1 - \frac{|\m L|}{b}\right)}.\]
If equality occurs, then for each $S' \in \{S, X \setminus S\}$ and each $\m L' \in \{\m L, \m B \setminus \m L\}$, every point of $S'$ is contained in a constant number of blocks of $\m L'$ and every block of $\m L'$ contains a constant number of points of $\m S'$ 
\end{lem}
\begin{rem}
See \cite[Theorem 5.1]{Haemers95_interlacing} or \cite[Theorem 4.9.1]{Brouwer-Haemers_book} where equality in the bound of Lemma \ref{lem:incidence_bound} is considered.
A proof of this bound that does not use any spectral techniques is given in \cite[Section 9]{Murphy-Petridis16}. 
\end{rem}

Several applications of this bound have been found \cite{WSV12, Lund-Saraf-Wolf16, Murphy-Petridis16, Vinh11}, but to our knowledge our result is the first one on blocking sets. 

\section{Proof of the upper bound}
\label{sec:proof}
The points and lines of a finite projective plane $\pi$ of order $n$ give rise to a $2$-$(n^2 + n + 1, n + 1, 1)$ design with $r = k = n + 1$ and $b = v = n^2 + n + 1$. 
Let $S$ be a minimal $t$-fold blocking set in $\pi$.  
For each point of $S$ choose a single $t$-secant of $S$ through that point (this can be done because $S$ is minimal), and let $\m L$ be the collection of these $t$-secants. 
We must have $t|\m L| \geq |S|$ since these $t$-secants cover all the points of $S$. 
So we will assume that $|\m L| = \lceil |S|/t \rceil$ by throwing away extra lines, if necessary.
The number of incidences between $S$ and $\m L$ is $t|\m L| = t \lceil |S|/t \rceil$.
Hence by Lemma \ref{lem:incidence_bound}, we obtain the following inequality for $x \coloneqq |S|$
\[\left\lvert t \left\lceil \frac{x}{t} \right\rceil - \frac{(n+1)x\lceil \frac{x}{t} \rceil}{(n^2 + n + 1)} \right\rvert \leq \sqrt{nx \left\lceil \frac{x}{t} \right\rceil \left(1 - \frac{x}{(n^2 + n+ 1)}\right)\left(1 - \frac{\lceil \frac{x}{t} \rceil}{(n^2 + n + 1)}\right)}.\]
Dividing both sides by $\lceil \frac{x}{t} \rceil$ yields
\[\left\lvert t - \frac{(n+1)x}{(n^2 + n + 1)} \right\rvert \leq \sqrt{nx \left(1 - \frac{x}{(n^2 + n+ 1)}\right)\left(\frac{1}{\lceil \frac{x}{t} \rceil} - \frac{1}{(n^2 + n + 1)}\right)}.\]
Since $\lceil \frac{x}{t} \rceil \geq \frac{x}{t}$ we can also write
\[\left\lvert t - \frac{(n+1)x}{(n^2 + n + 1)} \right\rvert \leq \sqrt{nx \left(1 - \frac{x}{(n^2 + n+ 1)}\right)\left(\frac{t}{x} - \frac{1}{(n^2 + n + 1)}\right)}.\]
Squaring both sides and then simplifying the above inequality, we get

\begin{equation}
\label{eq_quadratic}
x^2 - ((t - 1)n + 2t)x - t(n - t)(n^2 + n + 1) \leq 0.
\end{equation}
This implies that $x$ must lie within the range defined by the roots of the quadratic equation $x^2 - ((t - 1)n + 2t)x - t(n - t)(n^2 + n + 1) = 0$. 
The roots are \[\alpha = \frac{1}{2} \left((t - 1) n + 2t - n\sqrt{4tn - 3t^2 + 2t + 1}\right)\] and \[\beta = \frac{1}{2} \left((t - 1) n + 2t + n\sqrt{4tn - 3t^2 + 2t + 1}\right).\] 
The inequality $x \leq \beta$ proves the result.  \hfill \qed 

\begin{rem} 
Note that the only property of the set $S$ we have used is that through each point of $S$ there exists a $t$-secant to $S$. 
For $t = 1$, such sets are known as \textit{tangency sets} \cite{Bruen-Drudge99} or \textit{strong representative systems} \cite{Szonyi91}, and in fact all the proofs of the Bruen-Thas upper bound only use the property that through every point there exists a tangent. 
In particular, we have given a new proof of the fact that a polarity\footnote{a bijective incidence preserving map between points and lines which equals its inverse} of a finite projective plane of order $n$ can have at most $n \sqrt{n} + 1$ absolute points, proved by Baer \cite{Baer46} when $n$ is not a square and by Seib \cite{Seib70} when $n$ is a square. 
When there is exactly one tangent through each point of the set then it is called a semioval, for which we get the same upper bound of $n \sqrt{n} + 1$, originally proved by Thas \cite{Thas74}.

Tangency sets in $\mathrm{AG}(n,q)$ and $\mathrm{PG}(n,q)$ with respect to lines are related to finite field Nikodym sets, which are sets $S$ of points for which through each point $x$ of the underlying space there exists a line $L_x$ with the property that $L_x \setminus \{x\} \subseteq S$. 
Clearly the complement of a Nikodym set is a tangency set and hence upper bounds on  tangency sets imply lower bounds on  Nikodym sets. 
We refer to \cite{Lund-Saraf-Wolf16} for further discussion on the Nikodym sets. 
\end{rem}

\section{The case of equality in the bound}
\label{sec:equality}
Let $\pi$ be a projective plane of order $n$. 
Lemma \ref{lem:incidence_bound} also implies that if $S$ is a $t$-fold blocking set of size $bn + t$ in $\pi$, with $b = \frac{1}{2}(\sqrt{4tn  - (3t + 1)(t - 1)}  + t - 1)$, then the following must hold (along with the natural condition that $b$ is an integer): 
\begin{itemize}
\item for every point of $S$ there exists a unique element of $\m L$ which contains $S$,
\item every line of $\pi$ which is not in the set $\m L$ defined in Section \ref{sec:proof} intersects $S$ in a constant number of points.
\end{itemize}
Fix a point $x$ of $S$ and let $L$ be the unique $t$-secant of $S$ contained in $\m L$ and through $x$. 
Then the remaining $bn$ points of $S$ are covered by the remaining $n$ lines through $x$ in $\pi$, all of which intersect $S$ in a constant number of points. 
This constant must then be equal to $b$. 
Therefore, in case of equality every line intersects $S$ in $t$ points or $b + 1 = \frac{1}{2}(\sqrt{4tn  - (3t + 1)(t - 1)}  + t - 1) + 1$ points.

A set of type $(k_1, k_2)$ in a finite projective plane of order $n$ is a set $S$ of points with the property that each line intersects $S$ in exactly $k_1$ or $k_2$ points. 
These sets have been studied by several authors \cite{deFinis81, Penttila-Royle95, Tallini87} and they are related to various objects in finite geometry. 
We have just shown that the $t$-fold blocking sets which meet the bound in Theorem \ref{thm:main} are sets of type $(t, b + 1)$ where $b = \frac{1}{2}(\sqrt{4tn  - (3t + 1)(t - 1)}  + t - 1)$. 
Therefore, such $t$-sets must satisfy the necessary conditions given in \cite[Section 2]{Penttila-Royle95}. 
In particular, $b + 1 - t$ must divide $n$. 
\begin{lem} In case $n$ is a prime power $q$ the only possible values of $t$ for which equality can be reached are 
\begin{itemize}
\item $t=1$ when $q$ is a square, in this case $S$ is a unital in $\pi$. 
\item $t=q-\sqrt{q}$ when $q$ is a square, in this case $S$ is the complement of a subplane of order $\sqrt{q}$ in $\pi$.
\item $t=q$ for any $q$, in this case $S$ is the plane $\pi$ with one point removed.
\end{itemize} 
\end{lem}

\begin{proof}
From the expression for $b$ we obtain 
\[
b^2+b(1-t)-t+t^2=tn. \tag{\(\dagger\)}
\]
Also 
\[
b-t+1 \text{ divides } n, \tag{\(\star\)}
\] see e.g. Section 2 of \cite{Penttila-Royle95}.
Assume now that $n=q=p^k$, then by $(\star)$ we have $b-t+1=p^h$. Write $t=\alpha p^l$, with $(\alpha,p)=1$.
 Hence $(\dagger)$ becomes
\begin{equation} \label{eqn}
p^h(p^h+\alpha p^l-1)-\alpha p^l+\alpha^2 p^{2l} = \alpha p^{l}q
\end{equation}

By $(\dagger)$ $t$ is a divisor of $b(b+1)=(p^h-1+t)(p^h+t)$, hence $t=\alpha p^l$ is a divisor of $p^h(p^h-1)$, so $l\leq h$ and $\alpha$ divides $p^h-1$ $(\ast)$. We will distinguish four cases.

{\bf Case I} $l>0,h>0,l<h$: dividing (\ref{eqn}) by $p^l$ implies that $p$ divides $\alpha$, a contradiction.

{\bf Case II} $l>0,h>0,l=h$: Then after division of (\ref{eqn}) by $p^h$ we obtain
\begin{equation}\label{eqn2}
(p^h+\alpha p^h-1)-\alpha+\alpha^2 p^{h} = \alpha q
\end{equation}
So $\alpha = \beta p^h-1$ and by (*) $\alpha = p^h-1$, and (\ref{eqn2}) yields after simplification that $p^{2h}=q$, hence $q$ is a square, $t=q-\sqrt{q}$ and $b=q-1$, and $S$ is the complement of a Baer subplane.

{\bf Case III} $l=0,h>0$: Then from (\ref{eqn}) we get $p^h(p^h+\alpha-1)-\alpha+\alpha^2 = \alpha q$ which implies that $\alpha^2-\alpha$ is a multiple of $p^h$. Since $(\alpha,p)=1$ we must have $\alpha = \beta p^h+1$. By $(\ast)$ $\beta=0$, so $\alpha=1$ implying $p^{2h}=q$ and hence $q$ is a square, $t=1$ and $b=\sqrt{q}$, and $S$ is a unital.

{\bf Case IV} $l=h=0$: Then we have $b=t=q$, and $S$ is the projective plane with one point removed. \qedhere

\end{proof}

This completes the proof of Theorem \ref{thm:main}. 

\section{An alternative proof of the upper bound}
\label{sec:alternate}

In this section we give an alternate proof of the bound in Theorem \ref{thm:main}, using the so-called \textit{variance trick}. 

Again let $S$ be a minimal $t$-fold blocking set in a finite projective plane of order $n$. 
Let $|S| = bn + t$, where $b \coloneqq (|S| - t)/n$ is not necessarily an integer. 
For each $i$, let $l_i$ denote the number of lines which intersect $S$ in exactly $i$ points. 
Then by standard double counting arguments we have
\begin{equation}
\label{eq1}
\sum_{i = 1}^{n^2 + n + 1} l_i = n^2 + n + 1, 
\end{equation}
\begin{equation}
\sum_{i = 1}^{n^2 + n + 1} i l_i = |S|(n + 1), 
\end{equation}
\begin{equation}
\sum_{i = 2}^{n^2 + n + 1} i(i - 1)l_i = |S|(|S| - 1).
\end{equation}
Let 
\begin{equation}
\label{eq4}
s = \sum_{i = 1}^{n^2 + n + 1} (i - b - 1)^2l_i.
\end{equation}
Then we can compute $s$ as a function of $b, n$ and $t$ using these equations. 
Moreover, since $l_t \geq |S|/t$ (via the same argument as in Section \ref{sec:proof}), we have the inequality 
\begin{equation}
\label{eq5}
s \geq (t - b - 1)^2|S|/t = (t - b - 1)^2(bn + t)/t.
\end{equation}
Combining these equations, and simplifying, we get the following inequality
\begin{equation}
\frac{n(b + 1)(-b^2 + (t - 1)b + t(n + 1) - t^2)}{t} \geq 0,
\end{equation}
which implies that $b^2 - (t - 1)b + t^2 - t(n + 1) \leq 0.$
Therefore, $b \leq \frac{1}{2}(t - 1 + \sqrt{4tn  - (3t + 1)(t - 1)})$, and consequently $|S| = bn + t \leq \frac{1}{2} n\sqrt{4tn - (3t + 1)(t - 1)} + \frac{1}{2} (t - 1)n + t$.
Moreover,  if equality occurs then $b = \frac{1}{2}(t - 1 + \sqrt{4tn  - (3t + 1)(t - 1)})$ is an integer and from (\ref{eq4}) it follows that $l_i = 0$ for all $i \not\in \{t,  b + 1\}$. \qed

\section{Construction of a minimal $t$-fold blocking set}
\label{sec:construction}
In this section we construct a minimal $t$-fold blocking set of size $q \sqrt{q} + 1 + (t - 1)(q - \sqrt{q} + 1)$ in $\mathrm{PG}(2,q)$ for every square $q$ and $t \leq \sqrt{q} + 1$. 
While this is a factor of $\sqrt{t}$ away from the bounds that we have proved, to our knowledge this is the best (general) construction for these parameters. 

Consider a non-degenerate Hermitian variety $\m H$ in $\mathrm{PG}(2, q)$ and let $P$ be a point of $\m H$. 
Through $P$ there is a unique tangent $L$ to $\m H$ and each of the remaining $q$ lines through $P$ intersect $\m H$ in $\sqrt{q} + 1$ points. Pick $t - 1$ of these lines $L_1, \dots, L_{t - 1}$, where $t \leq \sqrt{q} + 1$.
With respect to the unitary polarity $\perp$ that defines $\m H$, we have $L_1^\perp, \dots, L_{t - 1}^\perp \in P^\perp = L$. 
Let $B = \m H \cup L_1 \cup \dots \cup L_{t - 1} \cup \{L_1^\perp, \dots, L_{t - 1}^\perp\}$. 
Then $B$ has size $q \sqrt{q} + 1 + (t - 1)(q - \sqrt{q} + 1)$. 
We will show that $B$ is a minimal $t$-fold blocking set. 

Every secant of $\m H$ intersects $B$ in at least $\sqrt{q} + 1 \geq t$ points. 
Let $M$ be a tangent to $\m H$ at the point $Q$.  
If $Q = P$, then $M = L$ and hence $M$ contains exactly $t$ points of $B$. 
If $Q$ is not contained in any $L_i$, then $M$ intersects each $L_i$ in a unique point, and thus $|M \cap B| = t$. 
If $Q$ is contained in $L_i$ for some $i$, then $M$ contains $L_i^\perp$, and moreover it intersects each $L_j$ for $j \neq i$, thus proving that $|M \cap B| = t$. Therefore, $B$ is a minimal $t$-fold blocking set of size $|B|=q \sqrt{q} + 1 + (t - 1)(q - \sqrt{q} + 1)$. 

\begin{rem}
For 2-fold blocking sets one can do slightly better, as pointed out to us by Francesco Pavese. 
Namely, consider a point $P$ of $\m H$ and the tangent $L$ through $P$. Remove $P$ and add all the other points of $L$, i.e. consider the set $S = \m H \cup L \setminus \{P\}$. Then $S$ is a minimal 2-fold blocking set of size $q\sqrt{q} + q$. Indeed, as any line not containing $P$ containing a point of $L$ as well as a point of $\m H$, lines through $P$ different from $L$ intersect $S$ in $\sqrt{q}$ points. It is minimal as any tangent line $T$ to a point $Q\neq P$ of $\m H$ intersects $S$ in exactly two points, namely $Q$ and $T\cap L$.
\end{rem}



\section{Minimal blocking sets in designs and higher dimensional spaces}
\label{sec:higher_dim}

Minimal blocking sets in a block design can be defined as sets of points which meet every block non-trivially and have the property that through each point there is a unique block which intersects the blocking set in exactly one point. 
Using the main incidence bound we can obtain an upper bound for the size of these objects as well. 

\begin{thm}
\label{thm:design}
Let $S$ be a minimal blocking set in a symmetric $2$-$(v, k, \lambda)$ design. 
Then \[|S| \leq \left(\frac{1 + \sqrt{k - \lambda}}{k + \sqrt{k - \lambda}}\right) v.\]
\end{thm}
\begin{proof}
For each point of $S$ pick a block which intersects $S$ in exactly one point. 
In this way we have $x \coloneqq |S|$ points, $x$ blocks and exactly $x$ incidences between them.
Applying Lemma \ref{lem:incidence_bound}, we get
\[|x - rx^2/b| \leq \sqrt{r - \lambda} \sqrt{x^2(1 - x/v)(1 - x/b)}.\]
Since $b = v$ and $r = k$ in a symmetric design, we can simplify this to \[x \leq \left(\frac{1 + \sqrt{k - \lambda}}{k + \sqrt{k - \lambda}}\right) v	,\]
assuming that $x \geq v/k$. 
The assumption is true as it can be shown by counting incidences between points of $S$ and all blocks of the design.
We have $xr \geq b$ since there are exactly $r = k$ blocks through each point of $S$ and each of the $b = v$ blocks must contain at least one point of $S$. 
\end{proof}

From this general result on block designs, we obtain the following result of Bruen and Thas on tangency sets with respect to hyperplanes in 
$\mathrm{PG}(n,q)$.
\begin{cor}[\cite{Bruen-Thas82, Thas74}]
\label{cor:Bruen-Thas}
Let $S$ be a set of points in $\mathrm{PG}(n, q)$ such that  through each point of $S$ there exists a tangent hyperplane to $S$. 
Then $|S| \leq 1 + q^{(n+1)/2}$. 
\end{cor}
\begin{proof}
The points and hyperplanes of $\mathrm{PG}(n, q)$ form a symmetric $2$-$(v, k, \lambda)$ design with $v = (q^{n+1} - 1)/(q - 1)$, $k = (q^n - 1)/(q - 1)$ and $\lambda = (q^{n-1} - 1)/(q - 1)$. 
After substituting these values and simplifying, we get the bound. 
\end{proof}

For $n=2,3$ the bound in Corollary \ref{cor:Bruen-Thas} is sharp, while for $n > 3$ it never is.
This follows from the work of Thas \cite{Thas74}, as an easy counting argument shows that in case of equality, every tangent line to the set lies in a tangent hyperplane. The matching constructions for $n=2$ and $n=3$ are those of a unital in $\mathrm{PG}(2,q)$ and an ovoid in $\mathrm{PG}(3,q)$ respectively. 

\begin{rem}
Another bound on these tangency sets can be obtained by looking at the linear code spanned by the characteristic vectors of the hyperplanes with respect to the points, over the characteristic field. 
It is well known that if $q = p^h$, then the dimension of this $\mathbb{F}_p$-linear code is equal to $(\binom{p + n - 1}{n - 1})^h + 1$ (see for example \cite[Theorem 1.1]{Blokhuis-Moorhouse95}). 
This is an upper bound on the size of a tangency sets since the tangent hyperplanes that we construct give rise to linearly independent vectors. 
This bound is asymptotically better when $p$ is fixed and $h$, $n$ vary. 
\end{rem}

It is natural to consider minimal multiple blocking sets in the point-hyperplane designs of $\mathrm{PG}(n,q)$ for $n > 2$ as well. 
Let $S$ be a $t$-fold blocking set with respect to hyperplanes in $\mathrm{PG}(n,q)$, and let $\theta_n(q) = (q^n - 1)/(q - 1)$, so that $\theta_{n+1}(q)$ is the number of points (or hyperplanes) in $\mathrm{PG}(n,q)$. 

For each point of $S$ choose a single $t$-hyperplane of $S$ through that point, and let $\m H$ be the collection of these $t$-hyperplanes. 
We must have $t|\m H| \geq |S|$ since these $t$-hyperplanes cover all the points of $S$. 
So we will assume that $|\m H| = \lceil |S|/t \rceil$ by throwing away extra hyperplanes if necessary.
The number of incidences between $S$ and $\m H$ is $t|\m H| = t \lceil |S|/t \rceil$.

Then for $x \coloneqq |S|$ Lemma \ref{lem:incidence_bound} gives
\[\left \lvert t \left\lceil \frac{x}{t} \right\rceil - \frac{\theta_n(q) x\lceil \frac{x}{t} \rceil}{\theta_{n+1}(q) t} \right \rvert  \leq \sqrt{\theta_n(q) - \theta_{n-1}(q)} \sqrt{x\left\lceil \frac{x}{t} \right\rceil \left(1 - \frac{x}{\theta_{n+1}(q)}\right) \left( 1 - \frac{\lceil \frac{x}{t} \rceil}{ \theta_{n+1}(q)} \right)}.\]

With the same argument as in Section \ref{sec:proof} this simplifies to

\begin{equation}
\label{eq:hyperplanes}
(\theta^2_n(q) - q^{n-1}) x^2 + (- 2 \theta_n(q) \theta_{n+1}(q) t + q^{n-1} (\theta_{n+1}(q) + t \theta_{n+1}(q))) x  + t^2 \theta^2_{n+1}(q) - q^{n-1} t \theta^2_{n+1}(q) \leq 0.
\end{equation}

Note that for $n = 2$ we retrieve our equation (\ref{eq_quadratic}) for $t$-fold blocking sets in the projective plane.
We do not solve this inequality to get the precise upper bound for general $n$ but we can easily deduce from equation (\ref{eq:hyperplanes}) that the bound is $O(t^{1/2}q^{(n+1)/2} + qt)$, where the second term  plays a role only when $t$ is large (note that $t$ can be as large as $\theta_n(q) = q^{n-1} + \cdots + q + 1$). 

For $n = 3$ and $t \leq q$, a minimal $t$-fold blocking set with respect to the planes of size $q^2 + (t-1)q$ can be obtained by taking an ovoid in $\mathrm{PG}(3,q)$, $t-1$ tangent lines through a fixed point $P$ of the ovoid, and taking all the points on these lines along with the points of the ovoid, except the point $P$. 
When $q$ is a square, the Hermitian variety gives rise to a minimal $(q\sqrt{q} + 1)$-fold blocking set of size $q^{5/2} + q^{3/2} + q + 1$ in $\mathrm{PG}(3,q)$. 
We do not explore the $n = 3$ case further and leave it as an open problem to study large minimal $t$-fold blocking sets with respect to hyperplanes in $\mathrm{PG}(n,q)$ for $n > 2$. 

The Gaussian coefficient \[{n + 1 \brack k + 1}_q = \frac{(q^{n + 1} - 1)\cdots(q^{n+1} - q^{n-k})}{(q^{k+1} - 1)\cdots(q^{k+1} - q^{k})}\]
counts the total number of $k$ dimensional subspaces of $\mathrm{PG}(n, q)$. 
Points and $k$-dimensional subspaces of $\mathrm{PG}(n, q)$ form a $2 \mh ({n + 1 \brack 1}_q, {k + 1 \brack 1}_q, {n - 1 \brack k - 1}_q)$ design, with $r = {n  \brack k}_q$ and $b = {n + 1 \brack k + 1}_q$.
We have already studied the symmetric design which one gets by taking $k = n - 1$. 
This design is non-symmetric for all $k < n - 1$, and the problem of large minimal blocking sets becomes much harder in that case. 
The following bound can be obtained for minimal blocking sets in non-symmetric designs.

\begin{thm}
Let $S$ be a minimal blocking set in a $2$-$(v, k, \lambda)$ design. 
Then \[|S| < \left( \frac{1 + \sqrt{r - \lambda}}{k}\right) v.\]
\end{thm}
\begin{proof}
For each point of $S$ pick a block which intersects $S$ in exactly one point. 
Then for $x := |S|$ the incidence bound gives us:
\[|x - rx^2/b| \leq \sqrt{r - \lambda} \sqrt{x^2(1 - x/v)(1 - x/b)} < x\sqrt{r - \lambda} ,\]
where the second inequality follows from the fact that $(1 - x/v)(1 - x/b) < 1$. 
This simplifies to 
$x \leq (1 + \sqrt{r - \lambda})r/b$ and since $r/b = v/k$, we get our result. 
\end{proof}

This bound is non-trivial only if $1 + \sqrt{r - \lambda} < k$. 
Thus, in particular, it can be shown that we do not obtain any non-trivial upper bound on the cardinality of minimal blocking sets with respect to $k$-spaces in $\mathrm{PG}(n, q)$ for any $k < n - 1$.
For $k = 1$ and $n = 3$, Lund, Saraf and Wolf have proved a non-trivial upper bound using Lemma \ref{lem:incidence_bound} in a slightly different way \cite{Lund-Saraf-Wolf16}. 
It should also be noted that the lower bounds on Nikodym sets obtained using the polynomial method (see \cite{Lund-Saraf-Wolf16} and the references therein) give us non-trivial upper bounds on minimal blocking sets in $\mathrm{PG}(n, q)$ with respect to lines.

\section{Further generalization}
\label{sec:further}
Instead of looking at sets of points in which through each point there is at least one $t$-secant, we can look at the more general scenario where through each point there are at least $s$ lines which are $t$-secants, for some integer $s \geq 1$. 
We can apply the main incidence bound for this problem as well and get the following result. 

\begin{thm}
\label{thm:gen}
Let $S$ be a set of points in a finite projective plane of order $n$ with the property that through each point of $S$ there exist at least $s$ lines intersecting $S$ in exactly $t$ points. 
Then $|S|$ is at most the larger root of the quadratic equation \[sx^2 - (2st(n+1) - (s+t)n)x - t(n - st)(n^2 + n + 1) = 0.\]
\end{thm}
\begin{proof}
Let $x \coloneqq |S|$.
For each point of $S$ pick $s$ $t$-secants to obtain a set of $s\lceil |S|/t \rceil$ lines with the total number of incidences between $S$ and these lines equal to $st \lceil |S|/t \rceil$. 
Substituting these values in Lemma \ref{lem:incidence_bound} and simplifying as before, we obtain the quadratic inequality
\[sx^2 - (2st(n+1) - (s+t)n)x - t(n - st)(n^2 + n + 1) \leq 0.\]
Therefore, $x$ lies in the range defined by the two roots of the equation, and in particular $x$ is less than or equal to the larger root. 
\end{proof}

The solutions of the quadratic equation in Theorem \ref{thm:gen} look a bit messy and not particularly insightful. 
But for the case of $t = 1$, we directly get the following result related to the semiarcs.

\begin{cor}[{\cite[Theorem 2.4]{Csajbok12}}]
Let $S$ be a set of points in a finite projective plane of order $n$ with the property that through each point of $S$, there exist at least $s$ lines intersecting $S$ in exactly $1$ point. 
Then \[|S| \leq 1 + \frac{n}{2s}\left((s - 1) + \sqrt{4sn - 3s^2 + 2s + 1}\right).\]
\end{cor}

\begin{rem}
Similar results can be obtained for symmetric block designs using Lemma \ref{lem:incidence_bound}. 
\end{rem}

\section{Open Problems}
\label{sec:open}

The main open problem is the discrepancy between the size of known examples and the size of the upper bound in the case where $q$ is not a square. We can split this into two subproblems.

\begin{enumerate}[(1)]
\item \textit{Improve on our upper bound in case $q$ is not a square.}

In the case $t=1$ an improvement on the upper bound is obtained when $q$ is not a square in \cite{Szonyi05}. We were unable to generalize their method as it already broke down in the first step and believe that new ideas are needed in order to obtain such an improvement. 

\item \textit{Find constructions of large minimal $t$-fold blocking sets when $q$ is not a square.}

It seems very challenging to find constructions of large minimal $t$-fold blocking sets. In particular when $q$ is not a square, or if we require that the set does not contain any full lines, good examples remain to be found. In particular even for the particular case $t=q-1$, finding a large minimal $(q-1)$-fold blocking set is equivalent to finding to finding a small complete arc in $\mathrm{PG}(2,q)$ by taking complements. This problem is known to be hard and the best result is due to Kim and Vu who used the probabilistic method \cite{KimVu}.
Also note that the known lower bounds on complete arcs in $\mathrm{PG}(2,q)$ (see \cite{Ball97}) give us upper bounds on minimal $(q-1)$-fold blocking sets which are better than the bound that we obtain by taking $t = q - 1$ in Theorem \ref{thm:main}. 
\end{enumerate}

Even when $q$ is a square, the largest construction that we have for $t \not \in \{1, q - \sqrt{q}, q\}$ is far away from the upper bound. 
This leads to the third open problem.
\begin{enumerate}[(3)]
\item \textit{Find minimal $t$-fold blocking sets in $\mathrm{PG}(2,q)$ for $q$ square and $t \not\in \{1, q - \sqrt{q}, q\}$ which are larger than the one constructed in Section \ref{sec:construction}.} 
\end{enumerate}

Finally as mentioned in Section \ref{sec:higher_dim} it would be interesting to 

\begin{enumerate}[(4)]
\item \textit{Study large minimal t-fold blocking sets with respect to hyperplanes in $\mathrm{PG}(n, q)$ for $n > 2$.}
\end{enumerate}

\begin{thebibliography}{99}

\bibitem{Baer46}
R.~Baer.
\newblock Polarities in finite projective planes. 
\newblock {\em Bull. Amer. Math. Soc.} 52:77--93, 1946. 
%
\bibitem{Ball96}
S.~Ball.
\newblock Multiple blocking sets and arcs in finite planes.
\newblock {\em J. London Math. Soc.}, 54:581--593, 1996.

\bibitem{Ball97}
S.~Ball.
\newblock On small complete arcs in a finite plane.
\newblock {\em Discrete Math.}, 174:29--34, 1997.


\bibitem{Ball_book}
S.~Ball.
\newblock {\em Finite Geometry and Combinatorial Applications}, volume~82 of
  {\em London Mathematical Society Student Texts}.
\newblock Cambridge Univ. Press, Cambridge, 2015.


\bibitem{Barwick-Ebert_book}
S.~Barwick and G.~Ebert.
\newblock {\em Unitals in Projective Planes}.
\newblock Springer, 2008.

\bibitem{Blokhuis96_survey}
A.~Blokhuis.
\newblock Blocking sets in {D}esarguesian planes.
\newblock In {\em Combinatorics: Paul Erd\H{o}s is Eighty, Vol. 2}, pages
  133--155. Bolyai Soc. Math. Stud, 1996.
  
\bibitem{Blokhuis-Moorhouse95}
A.~Blokhuis and G. Eric Moorhouse.
\newblock Some $p$-Ranks Related to Orthogonal Spaces. 
\newblock {\em J.~Algeb.~Combin.}, 4:295--316, 1995. 

\bibitem{DeBeule-Storme12}
A.~Blokhuis, P.~Sziklai and T.~Sz\H{o}nyi.
\newblock Blocking sets in projective spaces. 
\newblock In: J.~De Beule and L.~Storme (eds.), {\em Current research topics on Galois geometry}, pages~61--84.
\newblock Nova Academic Publishers, 2011.

\bibitem{Brouwer-Haemers_book}
A.~E. Brouwer and W.~Haemers.
\newblock {\em Spectra of Graphs}.
\newblock Springer, 2012.

\bibitem{Bruen-Drudge99}
A.~A. Bruen and K.~Drudge.
\newblock The return of the {B}aer subplane.
\newblock {\em J. Comb. Theory Ser. A}, 85:228--231, 1999.

\bibitem{Bruen-Thas77}
A.~A. Bruen and J.~A. Thas.
\newblock Blocking sets.
\newblock {\em Geom. Dedicata}, 6:193--203, 1977.

\bibitem{Bruen-Thas82}
A.~A. Bruen and J.~A. Thas.
\newblock Hyperplane coverings and blocking sets.
\newblock {\em Math. Z.}, 181:407--409, 1982.

\bibitem{Bruen83}
A.~A.~Bruen.
\newblock Arcs and multiple blocking sets.
\newblock In {\em Combinatoria, Symposia Mathematica}, 28:15--29,  1986.

\bibitem{Csajbok12}
B.~Csajb\'ok and G.~Kiss.
\newblock Notes on semiarcs.
\newblock {\em Mediterr. J. Math.}, 9:677--692, 2012.

\bibitem{deFinis81}
M.~de~Finis.
\newblock {\em On $k$-sets of type $(m,n)$ in projective planes of square
  order}, pages 98--103.
\newblock London Mathematical Society Lecture Note Series. Cambridge University
  Press, 1981.
  
  \bibitem{WSV12}
S.~De Winter, J.~Schillewaert, and J.~Verstraete.
\newblock Large incidence-free sets in geometries.
\newblock {\em Electron. J. Combin}, 19:\#P24, 2012.


\bibitem{Dembowski}
P.~Dembowski.
\newblock {\em Finite Geometries: Reprint of the 1968 edition}.
\newblock Classics in Mathematics. Springer, 1997.

\bibitem{Haemers95_interlacing}
W.~H. Haemers.
\newblock Interlacing eigenvalues and graphs.
\newblock {\em Linear Algebra Appl.}, 226:593--616, 1995.

\bibitem{HughesPiper73}
D.~Hughes and F.~Piper,
\newblock Projective planes. 
\newblock Graduate Texts in Mathematics, Vol. 6. Springer-Verlag, New York-Berlin, 1973. x+291 pp. 



\bibitem{Hoory-Linial-Wigderson06}
S.~Hoory, N.~Linial, and A.~Wigderson.
\newblock Expander graphs and their applications.
\newblock {\em Bull. Am. Math. Soc}, 43:439--561, 2006.

\bibitem{KimVu}
J.~Kim and V.~Vu,
\newblock Small complete arcs in projective planes. 
\newblock {\em Combinatorica 23} , no. 2, 311--363, 2003.


\bibitem{Lund-Saraf16}
B.~Lund and S.~Saraf.
\newblock Incidence bounds for block designs.
\newblock {\em SIAM J. Discrete Math.}, 30:1997--2010, 2016.

\bibitem{Lund-Saraf-Wolf16}
B.~Lund, S.~Saraf, and C.~Wolf.
\newblock Finite field {K}akeya and {N}ikodym sets in three dimensions.
\newblock {\em arXiv preprint \url{https://arxiv.org/abs/1609.01048}}, 2016.



\bibitem{Murphy-Petridis16}
B.~Murphy and G.~Petridis.
\newblock A point-line incidence identity in finite fields, and applications.
\newblock {\em Moscow Journal of Combinatorics and Number Theory}, 6:64--95, 2016.

\bibitem{Penttila-Royle95}
T.~Penttila and G.~Royle.
\newblock Sets of type $(m, n)$ in the affine and projective planes of order
  nine.
\newblock {\em Des. Codes Cryptogr.}, 6:229--245, 1995.

\bibitem{Seib70} 
M.~Seib.
\newblock Unit\"are Polarit\"aten endlicher projektiver Ebenen. (German) 
\newblock {\em Arch. Math.} 21:103--112, 1970. 


\bibitem{Szonyi91}
T.~Ill\'{e}s, T.~Sz{\H{o}}nyi and F.~Wettl.
\newblock Maximal strong representative systems and minimal blocking sets.
\newblock {\em Mitt. Math. Sem. Giessen} 201:97--107, 1991. 

\bibitem{Szonyi92}
T.~Sz{\H{o}}nyi.
\newblock Note on the existence of large minimal blocking sets in Galois planes.
\newblock {\em Combinatorica}, 12:227--235, 1992.


\bibitem{Szonyi05}
T.~Sz{\H{o}}nyi, A.~Cossidente, A.~G\'acs, C.~Mengy\'an, A.~Siciliano, and
  Z.~Weiner.
\newblock On large minimal blocking sets in $\mathrm{PG}(2,q)$.
\newblock {\em J. Combin. Designs}, 13:25--41, 2005.

\bibitem{Tallini87}
G.~Tallini.
\newblock Some new results on sets of type $(m, n)$ in projective planes.
\newblock {\em J. Geom.}, 29:191--199, 1987.

\bibitem{Thas74}
J.~A.~Thas.
\newblock On semi ovals and semi ovoids.
\newblock {\em Geom. Dedicata}, 3:229--231, 1974.

\bibitem{Vinh11}
L.~A.~Vinh.
\newblock The Szemer\'edi-Trotter type theorem and the sum-product estimate in finite fields.
\newblock {\em European J. Combin.}, 32:1177--1181, 2011. 


\end{thebibliography}

\textbf{Address of the authors}\\
Anurag Bishnoi,
Institut F\"ur Mathematik,
Freie Universit\"at Berlin,\\
Arnimallee 3,
14195 Berlin,
Germany\\
\texttt{anurag.2357@gmail.com}\\

Sam Mattheus,
Department of Mathematics,
Vrije Universiteit Brussel,\\
Pleinlaan 2,
B-1050 Elsene,
Belgium\\
\texttt{Sam.Mattheus@vub.ac.be}\\

Jeroen Schillewaert,
Department of Mathematics, University of Auckland\\
Private Bag 92019,
Auckland 1142,
New Zealand\\
\texttt{j.schillewaert@auckland.ac.nz}\\

\end{document}